\documentclass{amsart}
\usepackage{amsmath,amsfonts,amssymb}
\usepackage{pinlabel}

\newtheorem{theorem}{Theorem}[section]
\newtheorem{lemma}[theorem]{Lemma}

\newtheorem{corollary}[theorem]{Corollary}
\newtheorem{conjecture}[theorem]{Conjecture}
\newtheorem{question}[theorem]{Question}

\newtheorem*{IQLemma}{(Lem.~\ref{lem:EA}) IQ gcd Lemma}
\newtheorem*{QIQTheorem}{(Thm.~\ref{thm:inthull}) Integer Hull Theorem}
\newtheorem*{QRCorollary}{(Cor.~\ref{cor:QR}) Surgery Family Theorem}

\theoremstyle{definition}
\newtheorem{definition}[theorem]{Definition}

\theoremstyle{remark}
\newtheorem{remark}[theorem]{Remark}
\newtheorem{example}[theorem]{Example}

\numberwithin{equation}{subsection}

\def\N{\mathbb N}
\def\Z{\mathbb Z}
\def\R{\mathbb R}

\def\x{{\bf x}}

\def\cl{\textnormal{cl}}
\def\scl{\textnormal{scl}}
\def\Hom{\textnormal{Hom}}

\newcommand{\IQ}{\ensuremath{\mathrm{IQ}}}
\newcommand{\QIQ}{\ensuremath{\mathrm{QIQ}}}
\newcommand{\QR}{\ensuremath{\mathrm{QR}}}

\title{Integer hulls of linear polyhedra and scl in families}
\author{Danny Calegari}
\address{Department of Mathematics \\ Caltech \\
Pasadena CA, 91125}
\email{dannyc@its.caltech.edu}
\author{Alden Walker}
\address{Department of Mathematics \\ Caltech \\
Pasadena CA, 91125}
\email{awalker@caltech.edu}
\date{\today}
\begin{document}

\begin{abstract}
The \emph{integer hull} of a polyhedron is the convex hull of the integer points contained in it. We
show that the vertices of the integer hulls of a rational family of polyhedra of size $O(n)$
have eventually quasipolynomial coordinates. As a corollary, we show that the stable commutator length of 
elements in a surgery family is eventually a ratio of quasipolynomials, and that unit balls in the scl 
norm eventually quasi-converge in finite dimensional surgery families.
\end{abstract}

\maketitle

\section{Introduction}

Integer programming problems are ubiquitous in mathematics, computer science and operations research.
Though integer programming is NP hard in general, Lenstra \cite{Lenstra} famously showed that integer
programming {\em with a fixed number of variables} is polynomially solvable. A family of integer
programming problems associated to a constraint matrix of fixed form but variable entries falls within
the domain of Lenstra's algorithm. A matrix with entries which are functions of some parameter
determines a {\em family} of polyhedra --- each value of the parameter determines the polyhedron
which is the convex hull of the column vectors of the associated matrix. Lenstra's algorithm
shows that one can efficiently decide which of the polyhedra in such a family contain an integer
lattice point.

The {\em integer hull} of a polyhedron is the convex hull of the integer lattice points 
contained in it. The integer hull is itself a polyhedron, and may be described by enumerating its
extremal vectors. There is essentially no relationship between the number of extremal vertices
of a polyhedron and the number of extremal vertices of its integer hull --- even a triangle in
the plane can have an integer hull with arbitrarily many vertices. In fact, even estimating the
number of extremal vertices in the integer hull is an important and very difficult problem; see 
\cite{Zolotykh} for a survey. However in this paper we
show that for families of polyhedra $V(n)$ whose extremal vertices are rational functions of size
$O(n)$ in a single integer variable $n$, the integer hulls themselves form a 
{\em quasipolynomial} family --- i.e.\/ there is an integer $\pi$ so that for each
fixed residue class of $n$ mod $\pi$ the extremal vertices of the integer hulls of $V(n)$
are the columns of a matrix with entries which are (linear) integer polynomials in $n$. 
An example of such a family $V(n)$ is the set of integer dilates of a
fixed rational polyhedron; Ehrhart's theorem (see e.g.\/ \cite{Barvinok}, Chapter~18) 
says that for such a family the {\em number} of integer lattice points is a quasipolynomial, 
and recent generalizations of this theory focus on counting the number of integer points 
contained in families of polyhedra (see \cite{Beck,Chen}).  Instead, we focus on understanding 
the \emph{extreme} integer points; that is, the integer hull.  This point of view 
is obviously related to results in integer programming (e.g. \cite{Sturmfels_Thomas}).  
The \emph{nature} of our results has more in common with integer programming than counting 
lattice points, but the \emph{style} and language of our results and methods have more in 
common with the Ehrhart theory.

Integral polyhedra arise often in low-dimensional topology --- as the unit ball in the
Gromov/Thurston norm on homology (\cite{Thurston_norm}), as the Newton polygon of an 
$A$-polynomial (\cite{Cooper_et_al, Garoufalidis}), as the tropicalization
of a character variety (\cite{Fock_Goncharov}),
as the convex hull of the support on a page of a spectral sequence or of
a collection of monopole spin-$c$ classes (\cite{Juhasz}),
as the parameter space for weights on a train track (\cite{Thurston_diffeomorphisms}), and so on. Short exact sequences connecting such
polyhedral objects arise from Dehn surgery, and the behavior of certain invariants can sometimes
be described in terms of families of integer programming problems. Sometimes some of this data 
can be expressed in terms of {\em quasipolynomials}; for example, Garoufalidis \cite{Garoufalidis}
shows that the degree of the Jones polynomials of certain families of knots are given by
the values of a quadratic quasipolynomial. In this paper we give an application of our theorem to
the $2$-dimensional bounded cohomology of free groups (equivalently, to the {\em stable commutator
length} of free groups).

Stable commutator length measures the {\em simplest} surface (in terms of $-\chi/2$)
that rationally bounds a given homologically trivial $1$-manifold in a given space 
(it has an equivalent algebraic definition in terms of commutators in
the fundamental group of the space; see \S~\ref{scl_background_section}). 
By cutting up a surface into simple pieces,
a surface $S$ may be encoded (potentially in many ways)
as a vector $v(S)$ in a vector space, and the set of all (weighted)
surfaces with prescribed boundary can be encoded as the vectors in a certain rational polyhedral
cone $P$. Typically, many weighted surfaces might represent a given vector, and it is in general very 
difficult to compute the objective function $|v|:=\inf_{v(S)=v} -\chi(S)/2$.
When the target space is a wedge of tori, there is a way of representing surfaces as vectors so
that the function $|v|$ is {\em piecewise rational linear} on the cone of all possible vectors.
More precisely, one forms the {\em sail} --- i.e.\/ the boundary of the convex hull of $D+P$, 
where $D$ is the set of integer lattice points in the interior of certain faces of $P$. 
Then $|v|$ is the sum of a linear term, plus a multiple of the {\em Klein function} --- i.e.\/
the function which is linear on rays
and $1$ on the sail. Certain families of $1$-manifolds --- those arising as a {\em surgery family} ---
give rise to rational linear families of polyhedral cones $P$. Consequently our main theorem implies
that stable commutator length is a {\em ratio of quasipolynomials} on a surgery family.
This is described in detail in \S~\ref{sss_section}.

\subsection{Statement of results}

We now give a more detailed account of the contents of the paper. 

\medskip

In \S~\ref{quasipolynomial_section}
we establish some facts about the behavior of the Euclidean algorithm on quasipolynomials,
and on a slightly more general class of functions $\IQ$, which are the {\em integer valued}
functions of the form $p(n)/C$ where $p(n)$ is an integral quasipolynomial and $C\in \Z$.
The results in this section are quite elementary and probably well known to experts, but they might
be unfamiliar to readers with more of a background in topology or geometric group theory, and therefore
they are included for completeness.

The main result of this section is that the gcd of two $\IQ$'s agrees with
the values of an $\IQ$ of a certain form for $n\gg 0$.

\begin{IQLemma}
Let $\frac{s(n)}{S}, \frac{t(n)}{T} \in \IQ$.  Then there are 
$\frac{a(n)}{d_1}, \frac{b(n)}{d_2} \in \IQ$ such that for $n \gg 0$,
$\gcd\left(\frac{s(n)}{S}, \frac{t(n)}{T}\right) 
= \frac{a(n)}{d_1}\frac{s(n)}{S} + \frac{b(n)}{d_2}\frac{t(n)}{T}$.
\end{IQLemma}

As a corollary (Corollary~\ref{cor:HNF}), we deduce that both the upper triangular matrix and the 
unimodular multiplier in the Hermite normal form of a matrix with entries in $\IQ$ themselves 
have entries eventually in $\IQ$.

\medskip

In \S~3 we study integer hulls of families of polyhedra. A function $S$ from integers to
finite subsets of $\Z^d$ is said to be $\QIQ$ if there is some positive integer $\pi$ so that
for each $0\le a < \pi$ the set $S(\pi n+a)$ is equal to the set of column vectors of 
a matrix with integral polynomial entries (in the variable $n$).  The nomenclature is 
not particularly enlightening: it is intended to evoke something like ``quasi IQ''.

We then state and prove the main theorem of our paper:

\begin{QIQTheorem}
For $1\le i\le k$, let $v_i(n)$ be a vector in $\R^d$ whose coordinates are rational functions
of $n$ of size $O(n)$, and let $V_n = \{ v_i(n) \}_{i=1}^k$. Then for $n \gg 0$, 
the integer hull of $V_n$ is $\QIQ$.
In particular, after passing to a cycle, the coordinates of the vertices of the integer hull of $V_n$
are in $\Z[n]$ for $n\gg 0$.
\end{QIQTheorem}

Here the terminology {\em passing to a cycle} means restricting $n$ to any coset of some 
specific finite index subgroup of $\Z$.

\begin{example}
The definition of $\QIQ$ is awkward because the \emph{number} of extreme integer points 
can vary.  Morally, we want to show that the vertices of the 
integer hulls have coordinates which are elements of $\IQ$; 
but there is no straightforward way to put the vertices of different polyhedra into
families without first passing to a cycle.  Consider the following very 
simple example.  
Let $f_p(n) = \frac{2n+1}{4} = \frac{n}{2} + \frac{1}{4}$ and 
$f_m(n) = \frac{2n-1}{4} = \frac{n}{2} - \frac{1}{2}$, and define the polyhedron 
\[
V_n = \{ (f_m, f_m), (f_p, f_m), (f_p, f_p), (f_m, f_p) \}.
\]
This is a square 
of side length $1/2$ centered at the point $(n/2, n/2)$.  

Then the integer hull of $V_n$ is empty if $n$ is odd, but for $n$ even it consists
of a single point $(n/2,n/2)$. In particular, the number of extremal
points in $V_n$ is periodic, but if we pass to a cycle, the extremal vertices
lie in families indexed by the integers, that have quasipolynomial coordinates
(as a function of the index).
\end{example}

\medskip

Finally, in \S~4 we apply the Integer Hull Theorem to the computation of stable commutator length in {\em surgery families}.
If $G$ is a group, let $B_1^H(G)$ denote the real vector space of formal homologically trivial sums of 
conjugacy classes in $G$ modulo {\em homogenization}, i.e.\/ modulo the relation
$g^n = ng$ for each $g\in G$ and $n\in\Z$.
Let $G=*_i A_i$ and $H=*_i B_i$ be free products of free abelian groups. A collection of
homomorphisms $A_i \to B_i$ determines a homomorphism from $G$ to $H$. A sequence of homomorphisms
$\rho(n):G \to H$ of this form is a {\em surgery family} if for each $i$, the sequence of homomorphisms
from $A_i$ to $B_i$ is an affine map of $\Z$ to the free abelian group $\Hom(A_i,B_i)$. The images $w(n)$
of some fixed rational chain $w\in B_1^H(G)$ under a surgery family of homomorphisms constitute a
{\em surgery family} of chains $w(n) \in B_1^H(H)$.

With this notation, we have the following corollary, which was observed experimentally in certain
cases in \cite{CalegariSSS}:

\begin{QRCorollary}
If $w(n)$ is a surgery family, then $\scl(w(n))$ is eventually a ratio of two quasipolynomials in $n$.
\end{QRCorollary}

One can also consider a surgery family $\rho(n)$ applied to a fixed subspace $V$ of $B_1^H(G)$ of arbitrary
dimension; the images form a family of subspaces of $B_1^H(H)$, and the unit balls in the
$\scl$ norm (in the image) are a family of rational convex polyhedra. Pulling back by $\rho(n)^{-1}$ gives a
family of polyhedra in a fixed vector space $V$. We show that the vertices of these 
polyhedra are eventually quasirational (i.e.\/ ratios of quasipolynomials) and the unit balls themselves
quasiconverge. Corollary~\ref{cor:convergent} gives the precise statement.

\section{Quasipolynomials}\label{quasipolynomial_section}

In this section we introduce (integral) quasipolynomials, and the slightly more general
class $\IQ$ of functions, and establish some basic facts about how
they behave under (coordinatewise) Euclidean algorithm and gcd. These facts are elementary, 
but since they are likely to be unfamiliar to geometric group theorists,
we give full details. Basic references for the theory of quasipolynomials and how 
they arise in lattice point geometry are \cite{Stanley} \S~4.4, and \cite{Beck_Robins}.

\medskip

If $a$ is an integer, and $b$ is a positive integer, we use the notation $a\% b$ to denote the
remainder when dividing $a$ by $b$. That is, $a\%b = a-b\lbrace a/b\rbrace$, where $\lbrace\cdot\rbrace$
denotes fractional part.

\begin{definition}\label{def:quasipolyonimal}
A {\em quasipolynomial} is a function $p:\N\to \Z$ for which there is some least positive integer
$\pi$ and a finite collection of polynomials $p_i(n) \in \Z[n]$ so that 
$p(n) = p_{n\% \pi}(n)$.  We will call $\pi$ the {\em period} of the quasipolynomial.  
\end{definition}

Note that other authors allow quasipolynomials to have arbitrary range, and do not require
the polynomials $p_i(n)$ to be integral (in $\Z[n]$). 
If we need to make the distinction, we refer to the specific
class of functions appearing in Definition~\ref{def:quasipolyonimal} as {\em integral quasipolynomials}.

We are naturally led to enlarge the class of (integral) quasipolynomials slightly.

\begin{definition}
A function $q:\N \to \Z$ is $\IQ$ if it is of the form $q(n) = p(n)/C$,
where $p(n)$ is an integral quasipolynomial, and $C$ is a constant integer
such that $C$ divides $p(n)$ for all $n \ge 0$. The {\em period} of $q$ is the period of $p$.
\end{definition}

We usually write an element of $\IQ$ explicitly as a ratio $p(n)/C$ where by
convention $p(n)$ is an integral quasipolynomial.

The next lemma is elementary but useful, since it shows that after restricting to values of $n$
in some coset of $\pi\Z$, we can think of an element of $\IQ$ just as an integral polynomial 
(in $\Z[n]$).
We call such a restriction {\em passing to a cycle}.

\begin{lemma}
\label{lem:actuallyLinear}
Let $q$ be in $\IQ$. Then there is an integer $\pi$ so that for each integer $a$ the function 
$q(\pi n+a)$ is contained in $\Z[n]$.
\end{lemma}
\begin{proof}
Let $q(m)=p(m)/C$ where $p(m)$ is quasipolynomial with period $D$.  Set $\pi = CD$, and consider the function $p(nCD+a)/C$ as a function of $n$ for any integer $a$.  Since $p$ has period $D$, there is some (non-quasi-) polynomial $p'$ so that $p(nCD+a)/C = p'(nCD+a)/C$, and we may expand the polynomial over the sum and collect terms containing $nCD$, so 
\[
p'(nCD+a)/C = p''(n) + A/C
\]
Where $p''(n)$ is an integer polynomial (note that every term involving $n$ must contain a $C$, which cancels with the denominator), and $A$ is a constant involving $a$.  Since by assumption $p'(nCD+a)/C$ is an integer, we must have $A/C \in \Z$.  This expresses $p(nCD+a)/C$ as an integral polynomial in $n$, as desired.
\end{proof}

The following example illustrates how functions in $\IQ$ naturally occur:
\begin{example}

Let $a$, $b$, $c$ be positive integers. Then
$$q(n):= \lfloor (an+b)/c \rfloor = \frac{an+b - (an+b)\%c}{c}$$
is in $\IQ$.
\end{example}

The sum of two quasipolynomials is a quasipolynomial, whose period divides the lcm of the periods of
the two terms. Similarly, the product of two quasipolynomials is a quasipolynomial. The same
holds for elements of $\IQ$, so $\IQ$ is a ring.

\begin{lemma}
\label{lem:quotient}
Let $a(n)/C$ and $b(n)/D$ be in $\IQ$, and suppose they take positive values. 
Then there are $q(n)/Q$ and $R(n)/R$ in $\IQ$ so that
for each $n\gg 0$, the integer quotient and remainder of division of $a(n)/C$ by $b(n)/C$ are
$q(n)/Q$ and $r(n)/R$ respectively.
\end{lemma}
\begin{proof}
By passing to a cycle, it suffices to prove this when $a(n)$ and $b(n)$ are
polynomial and eventually positive. 

Assume, therefore, that we are given $\frac{b(n)}{D}$ with $b(n)$ polynomial and eventually positive.  
We will inductively show the existence of a quotient and remainder for divisor $\frac{b(n)}{D}$ for 
any $\frac{a(n)}{C}$ with $a(n)$ polynomial and eventually positive.  First, we may assume that $C=D$, 
since we can find a constant common denominator.  Write $a(n) = a_kn^k + a_{k-1}n^{k-1} + \cdots + a_0$ 
and $b_ln^l + b_{l-1}n^{l-1} + \cdots + b_0$, where by hypothesis $a_k$ and $b_l$ are both
strictly positive.  If $k<l$, then for $n\gg 0$ the quotient is zero 
and the remainder is $\frac{a(n)}{C}$, so we are done.  If $k=l$, there are two cases:  
\begin{enumerate}
\item{If $a_k < b_l$, then the quotient is $0$ and the remainder is $\frac{a(n)}{C}$.}
\item{If $a_k \ge b_l$, then the quotient is $a_k/b_l - e$, where $a_k/b_l$ is the integer quotient, 
and $e$ is $1$ if $a(n)$ is eventually less than $(a_k/b_l)b(n)$, and $0$ otherwise.}
\end{enumerate}
This proves the lemma in the case of polynomials of equal degrees.  

We now induct on the degree $k$ of $a(n)$, and the size of the largest coefficient $a_k$.
Assume that for any $\frac{c(n)}{C}$ with degree less than $k$, or degree equal to $k$ but with 
highest coefficient less than $a_k$, we can write 
$\frac{c(n)}{C} = \frac{q(n)}{Q}\frac{b(n)}{D} + \frac{r(n)}{R}$,
where $\frac{r(n)}{R}$ is eventually less than $\frac{b(n)}{D}$ (i.e.\/ this expression eventually gives
the integer quotient and remainder).

There are two cases: 
\begin{enumerate}
\item{If $a_k > b_l$, then observe that $(a(n) - (a_k/b_l-e)n^{k-l}b(n))/D$, where $e$ is as above, 
has smaller first coefficient than $a(n)$, so by induction we can decompose it into a quotient 
$\frac{q(n)}{Q}$ and remainder $\frac{r(n)}{R}$ by $\frac{b(n)}{D}$.  The quotient of 
$\frac{a(n)}{D}$ by $\frac{b(n)}{D}$ is then $(a_k/b_l-e)n^{k-l} + \frac{q(n)}{Q}$ and the 
remainder is $\frac{r(n)}{R}$.}
\item{If $a_k \le b_l$, then consider the expression
\[
\frac{a(n)}{D} - \frac{a_kn + a_{k-1} - (a_kn+a_{k-1})\%b_l}{b_l} \frac{b(n)}{D}n^{k-l-1}
\]
The first coefficient (after multiplying $\frac{a(n)}{D}$ by $\frac{b_l}{b_l}$ 
to get a common denominator $b_lD$) of this difference is: $b_la_k - b_la_k = 0$, 
so the difference has smaller degree.  By induction, we find a quotient and remainder 
as above, and thus have an expression in $\IQ$ for the quotient and remainder of $\frac{a(n)}{D}$ 
by $\frac{b(n)}{D}$.  Note that as we apply induction in this case, we will need to 
multiply the top and bottom of $\frac{b(n)}{D}$ by $b_l$ to get a common denominator, 
but we will not need to alter the expression being divided, so the first coefficient 
and degree of the expression being divided can never increase.}
\end{enumerate}
The proof follows.
\end{proof}

Since quotient and remainder are in $\IQ$, the Euclidean algorithm lets us express
the gcd of two elements of $\IQ$ in a simple way:

\begin{lemma} 
\label{lem:EA}
Let $\frac{s(n)}{S}, \frac{t(n)}{T} \in \IQ$.  Then there are $\frac{a(n)}{d_1}, \frac{b(n)}{d_2} \in \IQ$ 
such that for $n \gg 0$, $\gcd\left(\frac{s(n)}{S}, \frac{t(n)}{T}\right) 
= \frac{a(n)}{d_1}\frac{s(n)}{S} + \frac{b(n)}{d_2}\frac{t(n)}{T}$.
\end{lemma}
\begin{proof}
By Lemma \ref{lem:quotient}, each step of the Euclidean algorithm has output which 
is eventually $\IQ$ if the input is. So all we need to do is show that there is an 
upper bound on the number of steps in the Euclidean algorithm, which is 
independent of $n$.

We write:
\begin{align*}
\frac{s(n)}{S} & =  \frac{q_1(n)}{Q_1} \frac{t(n)}{T} + \frac{r_1(n)}{R_1} \\
\frac{t(n)}{T} & =  \frac{q_2(n)}{Q_2} \frac{r_1(n)}{R_1} + \frac{r_2(n)}{R_2} \\
\frac{r_1(n)}{R_1} & =  \frac{q_3(n)}{Q_3} \frac{r_2(n)}{R_2} + \frac{r_3(n)}{R_3} \\
\cdots            & 
\end{align*}
By the proof of Lemma \ref{lem:quotient}, either the degree of $r_1(n)$ is 
smaller than the degree of $s(n)$, or the degrees are possibly the same, 
but the denominator $R_1$ is the same integral polynomial as $S$ and 
the leading coefficient of $r_1(n)$ is less than the leading coefficient 
of $s(n)$.  Thus, by taking two steps of the Euclidean algorithm, we have 
exchanged $\frac{s(n)}{S}$ for something strictly simpler. This process must 
terminate, because although the denominator can increase, it can do 
so only if the degree decreases, and this can happen only finitely many 
times.  Therefore, there is a universal upper bound on the length of 
the Euclidean algorithm, and we can find a finite expression for 
the $\gcd$.
\end{proof}

We recall the definition of Hermite normal form (see e.g.\/ \cite{Cohen}, \S~2.4). A (not necessarily
square) integral matrix is in {\em Hermite normal form} if it is upper triangular, and if the
entries are all non-negative and, in each column, maximized on the diagonal (strictly maximized if
nonzero). If $M$ is an arbitrary integral matrix, there is a unimodular integral
matrix $U$ so that $M=UB$ and $B$ is in Hermite normal form; $B$ is unique (given $M$),
though $U$ might not be. 

\begin{remark}
In fact, there is some ambiguity in what is meant by finding the HNF of a matrix. It is more
usual to find a factorization $M=BU$ where $B$ is HNF and $U$ is unimodular. We are usually
interested in applying HNF to an $n\times m$ matrix of rank $m$ where $m\le n$.
\end{remark}

\begin{corollary}\label{cor:HNF}
Both the upper triangular matrix and the unimodular multiplier in the Hermite normal form 
of a matrix with entries in $\IQ$ have entries eventually in $\IQ$.
\end{corollary}
\begin{proof}
The row reduction and $\gcd$ calculations in the algorithm to reduce to Hermite normal form 
(see \cite{Cohen}) make sense over the ring of functions which are $\IQ$ for $n\gg 0$,
by Lemma~\ref{lem:EA}.
\end{proof}

\section{Integer Hulls}

This section is the technical heart of the paper. The main result is Theorem~\ref{thm:inthull},
which says that certain families of polyhedra have integral hulls whose extremal vertices
have coordinates which are in $\IQ$ for $n\gg 0$. 
To state this precisely is a bit fiddly, since the number
of extremal vertices of the integral hull is itself variable.

We stress that the results we prove in this section generally hold {\em for all sufficiently large values}
of an integer parameter $n$. By abuse of notation, we will sometimes use the terms
integer polynomial or $\IQ$ as shorthand, when we really mean a function of $n$ which
agrees with some integer polynomial or $\IQ$ for $n\gg 0$. Generally the meaning should be
clear from the context, but we add the caveat ``for $n\gg 0$'' whenever there is the possibility
of ambiguity.

\begin{definition}
Given a finite set of points $P$ in $\R^d$, the \emph{integer hull} of $P$ is the convex hull 
of the integral points contained in the convex hull of $P$.
\end{definition}
Note that there is no assumption that the convex hull of $P$ has full dimension.

\begin{definition}
\label{def:QIQ}
A family of integer hulls $S(n)$ in $\R^d$ is $\QIQ$ if there is some integer $\pi$ 
so that for all integers $i$, the vertices of $S(\pi n+i)$ are the columns of a 
matrix whose entries are integer polynomials in the variable $n$, for $n\gg 0$.
\end{definition}

Note that for each $i$ the matrix whose columns are the vertices of $S(\pi n+i)$ is an 
$m_i\times d$ matrix, where $m_i$ depends on $i \mod \pi$. 
In the sequel we use the notation $S(n)=\{p_i(n)\}_{i=1}^k$
to mean that $S(n)$ is a set of $k$ vectors whose coordinates are functions of $n$ (taking
values in some ring; e.g.\/ they might be integral polynomials, $\IQ$, rational functions, etc.).
Note in this example that $k$ might depend on the period $\pi$, but in our notation we
implicitly pass to a cycle in which $k$ is constant.

\begin{lemma}
\label{lem:convHull}
Let $S(n) = \{p_i(n)\}_{i=1}^k$ and $R(n) = \mathrm{cone}(\{r_i(n)\}_{i=1}^m)$ where the vectors
$p_i(n)$ and $r_i(n)$ have coordinates which are rational functions of $n$.  
Then for $n\gg 0$, the vertices of the convex hull of $S(n) + R(n)$ are $\{p_j(n)\}$ for some
specific fixed subset of indices $j$.
\end{lemma}
\begin{proof}
Run the algorithm QuickHull (see \cite{quickhull} for a precise description)
on the points in $S(n)$ plus the rays in $R(n)$.  This algorithm repeatedly sorts points
based on their distances from hyperplanes constructed at each stage; with rational function
input, these hyperplanes have coordinates which themselves are rational functions.
The order of a set of distances of points whose coordinates are
rational functions to hyperplanes defined by rational functions 
eventually stabilizes, so the decision tree as the algorithm runs eventually stabilizes 
(and thus picks out a fixed subset of the $p_i$ as vertices).
\end{proof}

\begin{lemma}
\label{lem:intersection}
Let $H$ be a constant vector with rational coordinates, let $c(n)$ be a rational function of $n$ of size
$O(n)$, and let $d(n)$, $e(n)$ be vectors whose coordinates are rational functions of $n$ of size $O(n)$. 
Define a family of hyperplanes by the constraints
$H\cdot \x = c(n)$. The intersection of this hyperplane with a line segment of the form
$d(n) + t\cdot e(n)$ for $t \in [0,1]$ is either eventually empty, or is eventually a
rational point whose coordinates are rational functions of $n$ of size $O(n)$.
\end{lemma}
\begin{proof}
Solving for $t$ gives
$$H \cdot (d(n) + t e(n)) = c(n) \quad \text{ or equivalently, } \quad t = \frac{c(n) - H\cdot d(n)}{H\cdot e(n)}$$
Thus, $t$ is a rational function of $n$.  If it is anything larger than $O(1)$, it is eventually outside the interval $[0,1]$, and the intersection is empty.  If it is $O(1)$, then the intersection is $O(n)$, as desired.
\end{proof}

We are now in a position to prove our main theorem. 

\begin{theorem}
\label{thm:inthull}
For $1\le i\le k$, let $v_i(n)$ be a vector in $\R^d$ whose coordinates are rational functions
of $n$ of size $O(n)$, and let $V_n = \{ v_i(n) \}_{i=1}^k$. Then for $n \gg 0$, 
the integer hull of $V_n$ is $\QIQ$.
In particular, after passing to a cycle, the coordinates of the vertices of the integer hull of $V_n$
are in $\Z[n]$ for $n\gg 0$.
\end{theorem}

Note that the last claim follows from the first, since after passing to a cycle,
a function in $\IQ$ of size $O(n)$ is integral linear.

Before proving this theorem, we explain the main idea. By subdivision, it suffices
to prove the theorem when the convex hull of $V_n$ is a simplex. Then the proof proceeds
by induction. Roughly speaking, there are two cases. In the first case, the 
simplex is either almost degenerate or not full dimensional.  In this case, every integer point in 
$V_n$ is contained in one of finitely many affine hyperplanes. We reduce to the intersection
of the simplex with one of these hyperplanes, and induct on the dimension.  
In the second case, the simplex is full-dimensional, 
and we show that for each vertex $v_i$, there is an integer point $p_i$ inside the simplex and
uniformly close to $v_i$.  The convex hull of the $p_i$ contains ``almost'' every integer point 
in $V_n$, and all that remains to find the rest is to slice the thin ``shell'' between 
the convex hull of the $p_i$ and the simplex $V_n$ (see Figure~\ref{fig:decomp} for a 
preview).  In this case, too, we reduce the problem to a finite number of lower-dimensional 
polyhedra and induct.  It is worth noting that the base case of a single point 
is nontrivial and requires Lemma~\ref{lem:quotient}.

This bears more than a superficial similarity to Lenstra's algorithm in \cite{Lenstra}, mentioned in
the introduction. That algorithm takes as input a polyhedron with a fixed number of vertices 
and decides whether the integer hull is empty or not. The algorithm proceeds by either showing 
that such a vertex must exist (if some easily verified inequalities are satisfied) or reduces 
the search to the intersection of the polyhedron with finitely many affine hyperplanes.

\begin{proof}
Since the coordinates of each $v_i$ are rational functions of $n$, Lemma~\ref{lem:convHull} implies
that the vertices of the convex hull of $V_n$ are given by some fixed subset of the $v_i$.  
So after passing to a subset if necessary, we assume that the $v_i(n)$ are precisely the
vertices of the convex hull of $V_n$. 

We make the further simplification of decomposing $V_n$ into simplices whose vertices are the $v_i(n)$.
The integer hull of $V_n$ is the convex hull of the union of the integer hulls of these simplices.
Hence if we can show that the integer hull of each such simplex is $\QIQ$, then the integer hull
of $V_n$ is $\QIQ$, by another application of Lemma~\ref{lem:convHull}. So without loss of
generality we may assume that $V_n$ is simplicial; i.e.\/ $k$ is one greater than the dimension
of the convex hull of $V_n$. Furthermore, $k-1\le d$, since $V_n$ is $k-1$ dimensional.

By abuse of notation, we let $V_n$ denote both the polyhedron with vertices $v_i(n)$, and
the matrix with columns $v_i(n)$. We use matrix notation for the components of each vector $v_i(n)$;
hence $v_{j,i}(n)$ is the $j$th coordinate of the vector $v_i(n)$. Since each $v_{j,i}(n)$ is
$O(n)$, it can be expressed using Lemma \ref{lem:quotient} as the sum of a linear element of $\IQ$ plus a rational
function of size $O(1)$. Applying Lemma \ref{lem:actuallyLinear}, there is some $\pi\in\Z$
so that in the coset $\pi\Z$, we can write 
$v_{j,i}(n) = l_{j,i}n + O(1)$ where $l_{j,i}$ is a constant integer. By passing to a 
(possibly quite large) cycle, then, we can assume that the
matrix $V_n$ has the form $V_n = ln + O(1)$, where $l$ is a constant integer matrix.

Subtracting the integer column vector $l_1n$ from every column translates both $V_n$ and its
integer hull, so we may assume $l_1$ is the zero vector, and $v_1(n)$ has size $O(1)$. Now multiply
the matrix $l$ on the left by a (constant!) integer unimodular matrix $U$ to put it in Hermite normal 
form. The same unimodular matrix $U$ applied to $V_n$ puts it into an especially nice form; the first column
has size $O(1)$, and $v_{j,i}(n)$ has size $O(1)$ for $j\ge i$; see Figure~\ref{fig:matrix} for a
schematic picture of $V_n$ before and after left multiplication by the unimodular matrix $U$.

\begin{figure}[htpb]
\labellist
\small\hair 2pt
\pinlabel $O(1)$ at -13 83
\pinlabel $O(1)$ at 187 83
\pinlabel $O(n)$ at 313 83
\pinlabel $O(n)$ at 70 83
\pinlabel $O(1)$ at 250 60
\pinlabel $O(n)$ at 290 115
\endlabellist
\centering
\includegraphics[scale=1]{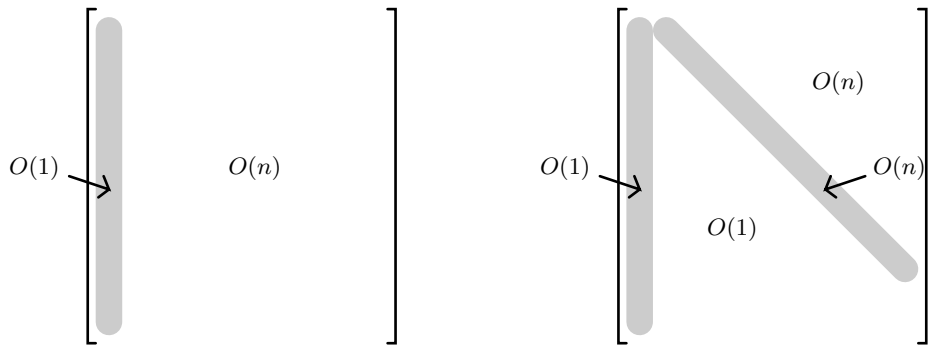}
\caption{The $d\times k$ matrix $V_n$ before and after left multiplication by a constant unimodular matrix 
taking the linear part to Hermite normal form. Note that after the multiplication, rows
$k$ through $d$ are of size $O(1)$, and the ``diagonal'' $v_{i,i+1}$ has size $O(n)$.}\label{fig:matrix}
\end{figure}

One of two things happens now.  Either rows $k$ through $d$ of $V_n$ are identically
zero (or nonexistent, if $k-1=d$), or there is some nonzero entry.  
If the rows are identically zero, then the convex hull of $V_n$ is contained
in the subspace $\R^{k-1}$ (where it is full dimensional, since by assumption $V_n$ is
a simplex). We will return to this case shortly.

Suppose one of the rows $k$ through $d$ contains some nonzero entry; without loss of generality
we may assume it is in row $k$. Since the entries in this row have size $O(1)$, there is a constant $C$
so that $|v_{k,i}(n)| \le C$ for all $i$ and all $n$. Hence the integer points in $V_n$ are
contained in the $2C+1$ affine hyperplanes whose $k$th coordinate is one of $-C,-C+1,\cdots,C$.

The intersection of the convex hull of $V_n$ with each of these affine hyperplanes consists of
polyhedra $P_{-C}(n)$, $P_{-C+1}(n)$, up to $P_C(n)$. Since by hypothesis row $k$ is not
identically zero, $V_n$ is not contained in a single affine hyperplane, and therefore the
dimension of each $P_i$ is strictly less than $k-1$. Moreover, by Lemma~\ref{lem:intersection}
the vertices of each $P_i$ are (for $n\gg 0$) defined by rational functions of size $O(n)$,
since the vertices of $P_i(n)$ are among the intersections of the $1$-dimensional
faces of the convex hull of $V_n$ with a (fixed) affine hyperplane.

Assume that the integer hulls of the polyhedra $P_i(n)$ are $\QIQ$. The integer
hull of $V_n$ is the convex hull of the integer hulls of the $P_i(n)$, so 
Lemma~\ref{lem:convHull} implies that this integer hull is $\QIQ$, proving the theorem in this case.
This gives the induction step (though not the base of the induction)
when $V_n$ is not full dimensional.

The base of the induction is when $V_n$ consists of a single point $p(n)$ in $\R^d$. Each
coordinate of $p(n)$ is a rational function in $n$. By Lemma~\ref{lem:quotient} after dividing
the numerator by the denominator, the remainder is in $\IQ$. An element of $\IQ$ either has positive
degree, in which case it is eventually nonzero, or it has degree zero, in which case it is
nonzero for some fixed set of residues mod the period $\pi$. So the integer hull of $V_n$
is $\QIQ$ in this base case.

\medskip

We now address the case of a full-dimensional simplex with vertices $V_n$,
of dimension $k-1$ in $\R^{k-1}$.  We will show that the integer hull is itself the convex hull 
of the union of integer hulls of finitely many lower-dimensional polyhedra, 
which by induction will imply the theorem.  This procedure is intuitively rather simple; 
we find integer points $p_i$ uniformly close to each vertex $v_i$ and then explain how
to find all integer vertices outside the convex hull of the $p_i$. The reader is invited
to read Example~\ref{exa:2dexample} which implements the algorithm in a simple, but nontrivial, case.

Recall that we have written $v = ln + O(1)$, then subtracted the column $l_1n$ from each column of
$v$, and multiplied the translated $v$ on the left by an integral unimodular matrix $U$ to put
$v$ in the form caricatured in Figure~\ref{fig:matrix}. If the span of $l$ is not full dimensional,
then by reordering the vectors before finding Hermite normal form, we may assume the $k-1$th
row has entries of size $O(1)$, and repeat the argument above. So we may assume the span of $l$
is full dimensional; in other words, we may assume that after subtracting $l_1$ from
each column and taking Hermite normal form, the $l_{i,i+1}$ are all strictly positive.
Hence as $n \to \infty$, the vertex $v_1$ converges to a finite point, and the convex hull of
the $v_i$ converges (in the Hausdorff topology) to a nondegenerate full-dimensional cone.
Such a cone necessarily contains an interior integer point, which we denote $p_1$. Note
that the size of the difference $\|p_1 - v_1(n)\|$ is {\em uniformly bounded}, independent of $n$.
After multiplying on the left by $U^{-1}$ and translating by $-l_1n$, we obtain an integral
vector $p_1(n)$ with linear integral coordinates, and with a uniform bound on $\|p_1(n) - v_1(n)\|$.

Repeating for each vertex, we obtain a constant $C$ and vectors 
$p_i(n)$ for each $i$ with linear integral coordinates so that $\|p_i(n) - v_i(n)\|< C$. 
Note that for $n\gg 0$ the set $P_n:=\lbrace p_i(n)\rbrace$ spans a simplex of full dimension.
This simplex is contained in the integer hull of $V_n$; we now show how to find the integer
points outside this simplex.

The difference of the convex hulls of $V_n$ and of $P_n$ is itself a polyhedron. This
polyhedron decomposes naturally into a union of prisms on the top dimensional faces,
and each prism can be subdivided into finitely many simplices in such a way that each
simplex has vertices in $V_n \cup P_n$, and has at least two vertices of the form
$p_i$, $v_i$ for some $i$. Standard methods of decomposing a prism into simplices with
these properties are used in elementary algebraic topology; see e.g.\/ the
proof of Thm.~2.10 in Hatcher \cite{Hatcher}. Note in particular that the combinatorics
of the subdivision is constant for $n\gg 0$.

In particular, each of the finitely many simplices in the subdivision has vertices
which are rational functions of $n$ of size $O(n)$, and contains a pair of vertices
which are distance $O(1)$ apart. In particular, the span of the linear part of each such simplex
is {\em not} full dimensional, and we can reduce to a previous case where the induction
hypothesis is assumed to hold. 

Hence the integer hull of each simplex in the convex hull of $V_n$ but outside the convex hull of 
$P_n$ is $\QIQ$. The integer hull of $V_n$
is the convex hull of the integer hulls of these finitely many simplices, together with 
the convex hull of $P_n$. By Lemma~\ref{lem:convHull} this is itself $\QIQ$. This completes the
induction step, and proves the theorem.
\end{proof}

\begin{example}\label{exa:2dexample}
We give an example of the theorem in the case of a full-dimensional polyhedron in $\R^2$ to 
illustrate the decomposition into 1-dimensional polyhedra.  

\begin{figure}[htpb]
\includegraphics[scale=1]{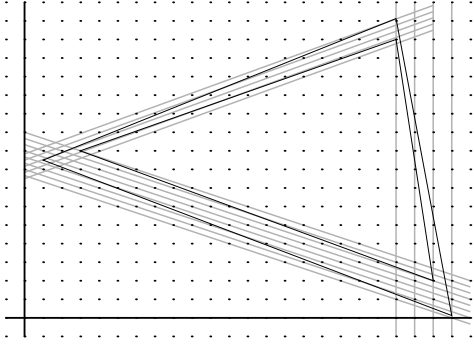}
\caption{The decomposition into polyhedra of dimension $1$ for $n=8$.}
\label{fig:decomp}
\end{figure}

Let 
\[
v_1 = \left[ \begin{array}{c} 2n+ \frac{1}{n} \\ 3n+1 \end{array} \right], \quad v_2 
= \left[ \begin{array}{c} n+\frac{1}{2} \\ 4 \end{array} \right], \quad v_3 
= \left[ \begin{array}{c} \frac{1}{n} \\ 3n+2 \end{array} \right]
\]
We find the $p_i$; first subtract the linear part of $v_1$ and obtain the Hermite normal form:
\[
\left[ \begin{array}{cc} -1 & 0 \\ -3 & 1 \end{array} \right] 
\left[ \begin{array}{ccc} \frac{1}{n} & -n + \frac{1}{2} & -2n + \frac{1}{n} \\ 
                          1          &   -3n + 4         &   2             
       \end{array} \right] = 
\left[ \begin{array}{ccc} -\frac{1}{n}    &  n - \frac{1}{2} & 2n-\frac{1}{n} \\ 
                          1-\frac{3}{n}  &  4 - \frac{3}{2} &  6n + 2 - \frac{6}{n}              
       \end{array} \right]
\]
So $v_1 \to (0, 1)$ (we will write column vectors as row vectors for simplicity), 
and the cone at $v_1$ limits to the cone spanned by $(1,0)$ and $(2,6)$.  
A fixed point in this cone is, for example, $(1,2)$, or $p_1 = (2n-1, 3n-1)$ 
in the original coordinates.  Similarly for $v_2$:
\[
\left[ \begin{array}{cc} 1 & 0 \\ -3 & 1 \end{array} \right] 
\left[ \begin{array}{ccc} n+\frac{1}{n} & \frac{1}{2} & -n + \frac{1}{n} \\ 
                          3n+1          &    4        &   3n-2             
       \end{array} \right] = 
\left[ \begin{array}{ccc} n+\frac{1}{n} & \frac{1}{2}    & -n + \frac{1}{n} \\ 
                        1 -\frac{3}{n} &  4-\frac{3}{2} &   6n-2 - \frac{3}{n}           
       \end{array} \right]
\]
So $v_2 \to (1/2, 5/2)$ and the cone limits to the cone spanned by $(1,0)$ and $(-1,6)$.  
A fixed point in this cone is $(1,3)$.  In the original coordinates, then, $p_2 = (n+1,6)$.  
A similar procedure gives $p_3 = (2,3n+1)$.

The remaining integral points are contained in three quadrilaterals, which can be decomposed into
six triangles with degenerate linear part. The integer lattice points in these simplices
are contained in finitely many additional affine linear subspaces, as indicated in Figure~\ref{fig:decomp}.
\end{example}

\begin{corollary}
\label{cor:coneHull}
Let $V$ be an cone with integral linear extremal vectors.  Then the integer hull (open or closed) of 
$V - 0$ is $\QIQ$.
\end{corollary}
\begin{proof}
The closed integer hull of such a cone is bounded by the hyperplane containing the extremal vectors, 
since the vectors are exactly linear (no error term), so the integer hull of the cone is the 
integer hull of this polyhedron, which is $\QIQ$ by Theorem \ref{thm:inthull}, plus the linear 
rays which are its extremal vectors. So in this case the proof follows from Lemma \ref{lem:convHull}.  
For the open cone, first observe that if $p-v \in V$ for some extremal vector $v$, then $p$ cannot 
be in the integer hull, so the integer hull is linearly bounded.  Then, simply change the bounds on 
the bounding hyperplanes by $1$ to slide them in slightly.  This gives a closed, linearly bounded 
polyhedron, which has a $\QIQ$ integer hull.
\end{proof}

\subsection{Relaxing the $O(n)$ assumption}

Given that the results in Section~\ref{quasipolynomial_section} do not require any $O(n)$ assumption, 
it is natural to ask why Theorem~\ref{thm:inthull} does.  We briefly address this question.  
In fact, we do not know whether the 
statement of the theorem is correct without the $O(n)$ assumption, but our proof does 
use it in a fundamental way.  Either the polyhedron is not full-dimensional or degenerate 
in some way, in which case we may slice it a uniformly bounded number of times to induct, or the 
polyhedron is full-dimensional.  In the full-dimensional case, we find an integer point uniformly 
close to each vertex and slice the ``shell'' of the polyhedron as shown in Figure~\ref{fig:decomp}.  
The $O(n)$ assumption is necessary to obtain the integer point uniformly close to each vertex.  
For example, consider the polyhedron in two dimensions with vertices
\[
v_1 = \left[ \begin{array}{c} \frac{1}{2} \\ \frac{1}{2} \end{array} \right]
\quad
v_2 = \left[ \begin{array}{c} n^2 \\ \frac{1}{2} \end{array} \right]
\quad
v_3 = \left[ \begin{array}{c} n^2 \\ n \end{array} \right]
\]
This polyhedron is \emph{not} degenerate, for its height grows linearly, but because 
its length grows like $n^2$, the polyhedron gets arbitrarily skinny close to the vertex $v_1$, 
and there is no integer point uniformly close to it.

This example is not definitive; 
one could hope to adapt Lenstra's key idea in \cite{Lenstra} to show that either
the highest order terms give a simplex of full dimension, or after acting
by a unimodular matrix, the simplex is decomposes into finitely many simplices of
smaller complexity (as measured either by dimension or degree of growth).

In a moment we shall explain why we believe that this hope can be realized in 
dimension $2$.
In dimensions $3$ and higher, intuition is hard to come by; we don't know how to 
prove the theorem, but we have no counterexample.

\begin{question}
Does Theorem~\ref{thm:inthull} hold without the $O(n)$ assumption?
\end{question}

In dimension $2$ there is a strong relationship between integer hulls and 
continued fractions, and this lead us to conjecture: 

\begin{conjecture}\label{conj:dim2}
In dimension $d=2$, Theorem~\ref{thm:inthull} holds without the $O(n)$ assumption.
\end{conjecture}

We give some explanation.  Given a line through the origin with slope $\alpha$, the 
convergents of the continued fraction approximation of $\alpha$ give lattice points on
alternating sides of the line which are the successively closest such points as measured
by the relative error of the approximation of $\alpha$ with respect to the horizontal
distance to the origin.  In particular, the integer 
hull of the cone spanned by $(0,1)$ and $(1, \alpha)$, without the origin, comprises the 
points $(p_i,q_i)$, where $p_i/q_i$ gives every other convergent of $\alpha$.  

It follows from Lemma~\ref{lem:EA} that the integer hull of a cone based at the origin 
with rays which are rational functions of $n$ is $\IQ$ (we do not include the origin 
in the integer hull --- otherwise, this would be trivial).

The analogue of Theorem~\ref{thm:inthull} in $2$ dimensions and without the $O(n)$ 
assumption does \emph{not} follow immediately.  It is true that we can recover the 
integer hull of a polyhedron from the integer hulls of the cones on each vertex, but 
we cannot reduce the integer hull of an arbitrary cone to the integer hull of a cone 
based at the origin: given an arbitrary cone, it is not difficult to see that 
we can act by a unimodular matrix and an integer shift to obtain an equivalent cone 
with rays $(1,0)$ and $(a,b)$, where $a,b\in \Z$, and where this equivalent cone is 
based at a point $(p/q,0)$, where $p,q \in \Z$ and $0 \le p/q < 1$.  Thus, we 
can \emph{almost} find an equivalent cone based at the origin, but not quite: it 
will be shifted over slightly.  Though small, this shift means that we cannot 
simply use Lemma~\ref{lem:EA} to conclude that the integer hull is $\IQ$.  

The integer hulls of cones as described above are very likely to have well-behaved 
integer hulls, and thus we make Conjecture~\ref{conj:dim2}, but we do not 
know how to prove it.

\section{Stable Commutator Length}

In this section, we discuss consequences of Theorem~\ref{thm:inthull} for geometric
group theory, in particular for the theory of
stable commutator length. A basic reference is \cite{CalegariSCL}, and our applications build on
the technology and viewpoint developed in \cite{CalegariSSS}.

\subsection{Background}\label{scl_background_section}

If $G$ is a group, and $g\in [G,G]$, the {\em commutator length} of $g$, denoted $\cl(g)$, is the
least number of commutators in $G$ whose product is $g$. The {\em stable commutator length} is
defined to be the limit $\scl(g):=\lim_{n\to \infty} \cl(g^n)/n$.

Stable commutator length extends to a function on homologically trivial formal finite sums of elements, 
by defining
$$\cl(g_1 + g_2 + \cdots + g_m) = \min_{\lbrace t_i \rbrace} \cl(g_1^{t_1}g_2^{t_2}\cdots g_m^{t_m})$$
(where superscript denotes conjugation) and
$$\scl(g_1 + g_2 + \cdots + g_m) = \lim_{n \to \infty} \cl(g_1^n + \cdots + g_m^n)/n$$
By linearity on rays, and continuity, $\scl$ extends to a pseudo-norm on the real vector space $B_1(G)$
of $1$-boundaries (in the sense of group homology), and vanishes on the subspace 
$H:=\langle g^n - ng, g - hgh^{-1}\rangle$. Hence $\scl$ descends to a pseudo-norm on $B_1^H(G):= B_1(G)/H$.
For certain groups --- e.g.\/ when $G$ is {\em word hyperbolic} --- $\scl$ is a genuine norm on $B_1^H$.

\medskip

Stable commutator length can be reformulated in topological terms. If $K$ is a $K(G,1)$, conjugacy
classes in $G$ correspond to free homotopy classes of loops in $K$. Let $g_1,g_2,\cdots,g_m \in G$
with $\sum g_i \in B_1$, and let $\Gamma:\coprod_i S^1_i \to K$ be the corresponding free homotopy classes. 
A map of a compact, oriented surface $f:S \to K$ is {\em admissible} (for $\Gamma$)
if no component of $S$ has positive Euler characteristic, and if the restriction
$\partial f:\partial S \to K$ factors as a product $\Gamma \circ i$ for some 
$i:\partial S \to \coprod_i S^1_i$ which is an
{\em oriented covering map} of degree $n(S)$. Then there is a formula
$$\scl(\sum g_i) = \inf_S -\chi(S)/2n(S)$$
where the infimum is taken over all surfaces admissible for $\Gamma$.

\subsection{Scl, sails and surgery}\label{sss_section}

In \cite{CalegariSSS} the problem of computing $\scl$ in a free product of free abelian groups is
reduced to a kind of integer programming problem. We summarize the conclusions of that paper; for
more details, see \cite{CalegariSSS}, especially \S~3--4. For simplicity we restrict attention to a product
of two factors $G=A*B$. 
Then we can take $K$ to be a wedge of tori $K(A,1)\vee K(B,1)$ and
consider collections of loops $\Gamma:\coprod_i S^1_i \to K$ which are obtained by concatenating
essential arcs alternately contained in the $K(A,1)$ and the $K(B,1)$ factor.
Note that we are implicitly restricting attention to a generic special case in which no $S^1_i$
maps entirely into $K(A,1)$ or $K(B,1)$; in the language of \cite{CalegariSSS} we are forbidding
{\em Abelian loops}. Removing this restriction would not add any serious technical difficulty, but
it would add considerably to the notation and interfere with the exposition.

The set of arcs mapping to $K(A,1)$ is denoted $T(A)$, and similarly for $B$. The homotopy class
of each arc can be identified with an element of $A = \pi_1(K(A,1))$, so we get a map
$h:T(A) \to A$. Let $C_1(A)$ be the real vector space spanned by $T(A)$.
Let $T_2(A)$ be the product $T(A) \times T(A)$, and let $C_2(A)$ be the
real vector space spanned by $T_2(A)$. There are linear maps $h:C_2(A) \to A\otimes \R$ defined
on basis elements by $h(\tau,\tau') = \frac 1 2(h(\tau) + h(\tau'))$, and $\partial:C_2(A) \to C_1(A)$ 
defined on basis elements by $\partial(\tau,\tau') = \tau-\tau'$. Let $V_A$ be the 
rational polyhedral cone in $C_2(A)$ of {\em non-negative} vectors in the kernel of $h$ and of $\partial$.

The {\em sail} of $V_A$ is defined to be the boundary of the convex hull of $D_A + V_A$, where
$D_A$ is the union of the integer lattice points contained in certain open faces of $V_A$
(precisely which faces are included depends on some additional combinatorial data associated to
$\Gamma$). For $v \in V_A$, let $|v|$ denote the $L_1$ norm of the vector $v$ (in the given basis).
Define the {\em Klein function} $\kappa$ on $V_A$ to be the non-negative function, linear on
rays, which is equal to $1$ on the sail. Note that since $V_A$ is a rational polyhedral cone,
the sail is a finite sided (noncompact) integral polyhedron, and $\kappa$ is the minimum of finitely
many rational linear functions.

\medskip

A surface $S$ mapping to $K$ with boundary factoring through $\Gamma$ somehow
decomposes into $S_A$ and $S_B$ mapping to $K(A,1)$ and
$K(B,1)$ respectively, and determines vectors $v(S_A)$ and $v(S_B)$ in $V_A$ and $V_B$ respectively.
The surfaces $S_A$ and $S_B$ have corners; they can be given an ``orbifold Euler characteristic''
$\chi_o$ in a natural way such that $\chi(S) = \chi_o(S_A) + \chi_o(S_B)$.

The following is a restatement of Lemma~3.9 from \cite{CalegariSSS}:
\begin{lemma}
Define $\chi_o:V_A \to \R$ by $\chi_o(v) = \kappa(v) - |v|/2$. There is an inequality
$\chi_o(v(S_A)) \ge \chi_o(S_A)$, and conversely, for any rational vector $v\in V_A$ and any
$\epsilon > 0$ there is an integer $n$ and a surface $S_A$ with $v(S_A)=nv$ and
$|\chi_o(S_A)/n - \chi_o(v)| \le \epsilon$.
\end{lemma}

Two surfaces $S_A$ and $S_B$ glue together to form an admissible surface $S$ if and only if the
vectors $v(S_A)$ and $v(S_B)$ lie in the intersection of $V_A \times V_B$ with a certain linear
subspace (defined by equating certain coordinates in $C_2(A)$ with coordinates in $C_2(B)$ according
to the combinatorics of $\Gamma$). This intersection is a polyhedral cone $Y$, and we can define
a function $\chi$ on $Y$ by $\chi(v,v') = \chi_o(v) + \chi_o(v')$. There is a linear map
$d:Y \to H_1(\coprod_i S^1_i)$ which takes a surface $S$ to the image of the fundamental class
$[\partial S]$ in the homology of $\coprod_i S^1_i$. The following lemma is a restatement of (part of)
Theorem~3.14 from \cite{CalegariSSS}:

\begin{lemma}
Given $C \in H_1(\coprod_i S^1_i)$ representing a
class in $B_1^H(A*B)$ there is a formula
$$\scl(C) = \min_{y \in d^{-1}(C)\cap Y} -\chi(y)/2$$
\end{lemma}

\subsection{Surgery families}
We recall the definition of surgery and surgery families from \cite{CalegariSSS}.  
Let $\{A_i\}$ and $\{B_i\}$ be families of free abelian groups, and $\rho_i : A_i \to B_i$ 
a family of homomorphisms.  Then $\rho_i$ induces a homomorphism $\rho: *_i A_i \to *_iB_i$, 
and we say $\rho$ is induced by {\em surgery}.  If $C \in B_1^H(*_iA_i)$, then $\rho(C)$ 
is obtained by surgery on $\rho$.  The nomenclature comes from {\em Dehn surgery} in
$3$-manifold topology (see \cite{CalegariSSS}). 
Now let $\sigma_i$ and $\tau_i$ be two families of homomorphisms as above.  
Define $\rho_i(p) = \sigma_i + p\tau_i$; $\rho_i$ is called a {\em line} of surgeries.

Let $w \in B_1^H(A'*B')$ and let $\rho(p): A'*B' \to A*B$ be a line of surgeries.  
Let $w(p) = \rho(p)(w)$.  We call $w(p)$ a {\em surgery family}.
One may think of a surgery family as an expression $a_1^{\alpha_1(p)}b_1^{\beta_1(p)}\cdots b_m^{\beta_m(p)}$
with $a_i \in A$ and $b_i \in B$,
where the exponents $\alpha_i$, $\beta_i$ are linear functions of $p$.

\begin{example}
The family $w(p) = aba^{-1}b^{-1}ab^{-p}a^{-1}b^p$ is a surgery family in $F_2$.
\end{example}

In a surgery family, the cones $V_A$ and $V_B$ vary, but the spaces $C_2(A)$ and $C_2(B)$ do
not. To understand the behavior of $\scl$ in surgery families we need one additional
ingredient; the following lemma is a restatement of (part of) \S~4.4 from \cite{CalegariSSS}:

\begin{lemma}\label{lem:surgery_matrix}
Let $w(p)$ be a surgery family, and let $M(p)$ be the integral matrix whose columns are
vectors spanning the extremal rays of $V_A(p)$. Then $M(p)=N+pN'$ where $N$ and $N'$ are fixed
integral matrices. In other words, the entries of $M(p)$ are integral linear functions of $p$.
\end{lemma}

\subsection{Stable commutator length in families}

\begin{definition}
\label{def:QR}
A function $r(n)$ is {\em quasirational} if there are quasipolynomials $p(n)$ and $q(n)$ 
such that $r(n) = p(n)/q(n)$.  A family of polyhedra $S(n)$ in $\R^d$ is
$\QR$ if there is some integer $\pi$ so that for all integers $i$, the
vertices of $S(\pi n+i)$ are the columns of a matrix whose entries are rational
functions in the variable $n$, for $n\gg 0$.
\end{definition}

In what follows, let $\{w_i(p)\}_{i=1}^k$ be a finite set of surgery families
as above, so that each $w_i(p)$ is in $B_1^H(G)$ for some free product
of free abelian groups $G$. For each $p$, 
define $W_p$ to be the span of the $w_i(p)$ in $B_1^H(G)$. Let us assume that
the $w_i$ are linearly independent for $p\gg 0$, so that $W_p$ is $k$-dimensional. We
would like to express a relationship between the restriction of $\scl$ to
the subspaces $W_p$ as a function of $p$. To do this, we introduce some
dummy symbols $\{e_i\}_{i=1}^k$ spanning a vector space $E$, and
define $I_p:W_p \to E$ for each $p$ by $I_p(w_i(p))=e_i$.

\begin{theorem}
\label{thm:families}
Let $\{w_i(p)\}_{i=1}^k$ be a collection of surgery families, as above, and let 
$B_p$ be the unit ball in the $\scl$ norm restricted to $W_p$.  
Then $I_p(B_p)$ is eventually $\QR$ in $E$; i.e.\/ the unit $\scl$ balls are eventually 
$\QR$, modulo the identification 
of vector spaces.
\end{theorem}
\begin{proof}
We are going to describe a modified algorithm which produces the $\scl$ unit ball in the positive 
orthant in $W_p$. Since $-C = C^{-1}$ in $B_1^H$, this is sufficient, since to obtain the full 
unit ball we simply run this algorithm $2^k$ times to get each orthant.  
After describing the new algorithm, we will show that it produces output which is $\QR$ in $E$.

For simplicity, we give the algorithm in the case of a free product of two free abelian groups; 
the generalization to more free factors is straightforward. Let $\Gamma$ be as above, so that the 
homotopy class of $\Gamma$ corresponds to the formal sum of
conjugacy classes in the support of $\sum_i w_i(p)$. Then $H_1(\Gamma)$ is free abelian, 
with rank equal to the sum of the number of terms in each $w_i(p)$ 
(in particular, this rank does not depend on $p$).

For each $p$ we construct $V_A$ and $V_B$ as cones in $C_2(A)$ and $C_2(B)$, and construct $Y$ as a 
subcone of $V_A \times V_B$ in $C_2(A) \times C_2(B)$. Note that $C_2(A)$ and $C_2(B)$ do not depend on $p$. 
Moreover, although $V_A$ and $V_B$ do depend on $p$, the cone $Y$ is obtained by intersecting their product 
with a fixed subspace of $C_2(A) \times C_2(B)$.

There is a map $H_1(\Gamma) \to H_1(K)$ given by inclusion, and there is
a sequence
$$Y \xrightarrow{d} H_1(\Gamma) \to H_1(K)$$
which is exact at $H_1(\Gamma)$ in the sense that the image of $Y$ is equal to the cone in the
kernel of $H_1(\Gamma) \to H_1(K)$ consisting of non-negative multiples of the components
of $\Gamma$. This kernel contains the span of the $w_i$ by hypothesis. Denote this span by $W_p$,
and define $Y_p$ to be the preimage of this span; i.e.\/ $Y_p = d^{-1}(W_p)$.

The function $-\chi/2$ is piecewise linear on $Y$ for each $p$, and therefore the set
$\chi_1 = \{v \in Y_p \, | \, -\chi(v)/2 \le 1\}$ is a polyhedron. Therefore the projection
$d(\chi_1)$ is a polyhedron, which is precisely the $\scl$ unit ball.

We now need to show that the image of this unit ball is $\QR$ in $E$. 
For $n\gg 0$ the combinatorics of the cones $V_A$ and $V_B$ is eventually constant, and the
set of faces that contribute vectors to $D_A$ or $D_B$ is also eventually constant.
By Lemma~\ref{lem:surgery_matrix} each face of $V_A$ has spanning vectors which are integral
linear functions of $p$, so the vertices of the sail are $\QIQ$ by Corollary~\ref{cor:coneHull} 
and Lemma~\ref{lem:convHull}. Therefore the linear functions defining $\kappa_A$ and $\kappa_B$
are quasirational functions of $p$, so the same is true for the function $\chi$ on
$Y_p$, so the vertices and rays in $\chi_1$ are $\QR$. The projection by $d$ composed with
$I_p$ is therefore $\QR$ in $E$.
\end{proof}

\begin{corollary}\label{cor:QR}
If $w(n)$ is a surgery family, then $\scl(w(n))$ is eventually a ratio of two quasipolynomials in $n$.
\end{corollary}
\begin{proof}
Apply Theorem \ref{thm:families} to the case of a single surgery family.
\end{proof}

\begin{corollary}
\label{cor:convergent}
If $\scl(w_i(p))$ is bounded below (uniformly in $p$) for all $i$, then 
the unit $\scl$ ball in $E$ is quasi-convergent, in the sense that there 
is some integer $\pi$ such that the $\pi$ unit $\scl$ balls in $E$ which are 
the images of the unit balls in $W_{q\pi+i}$ for
each residue $i$ mod $\pi$ converge as $q\to\infty$.
\end{corollary}
\begin{proof}
The vertices of the unit $\scl$ ball are $\QR$, so after passing to a 
finite cycle (and then to a further finite cycle), we assume the vertices 
are rational functions.  As $\scl$ is uniformly bounded below, all these vertices 
must be bounded rational functions, which necessarily converge.
\end{proof}

\begin{figure}[htpb]
\centering
\begin{minipage}[b]{0.4\linewidth}
\centering
\includegraphics*[scale=0.19]{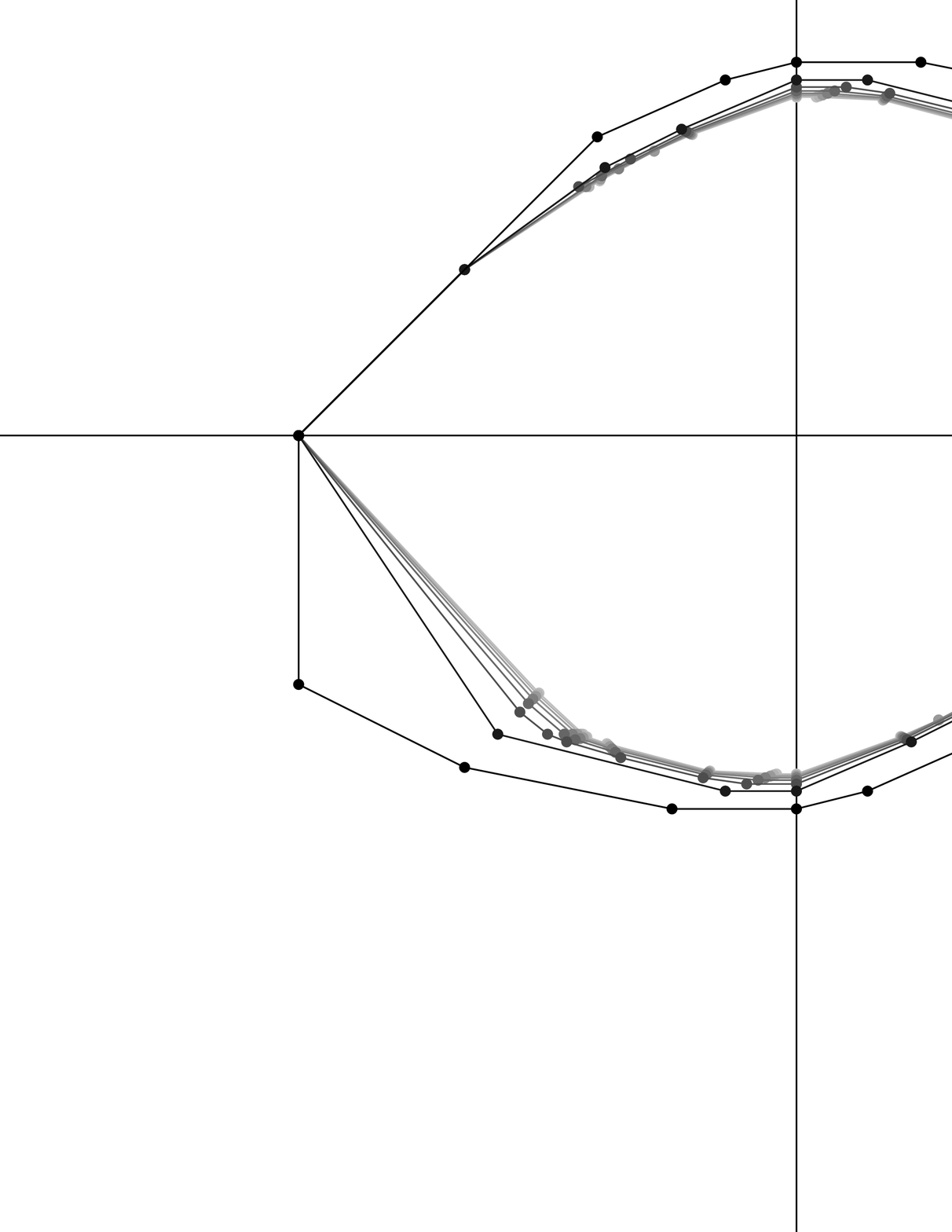}
\end{minipage}
\begin{minipage}[b]{0.59\linewidth}
\centering
\includegraphics*[scale=0.73]{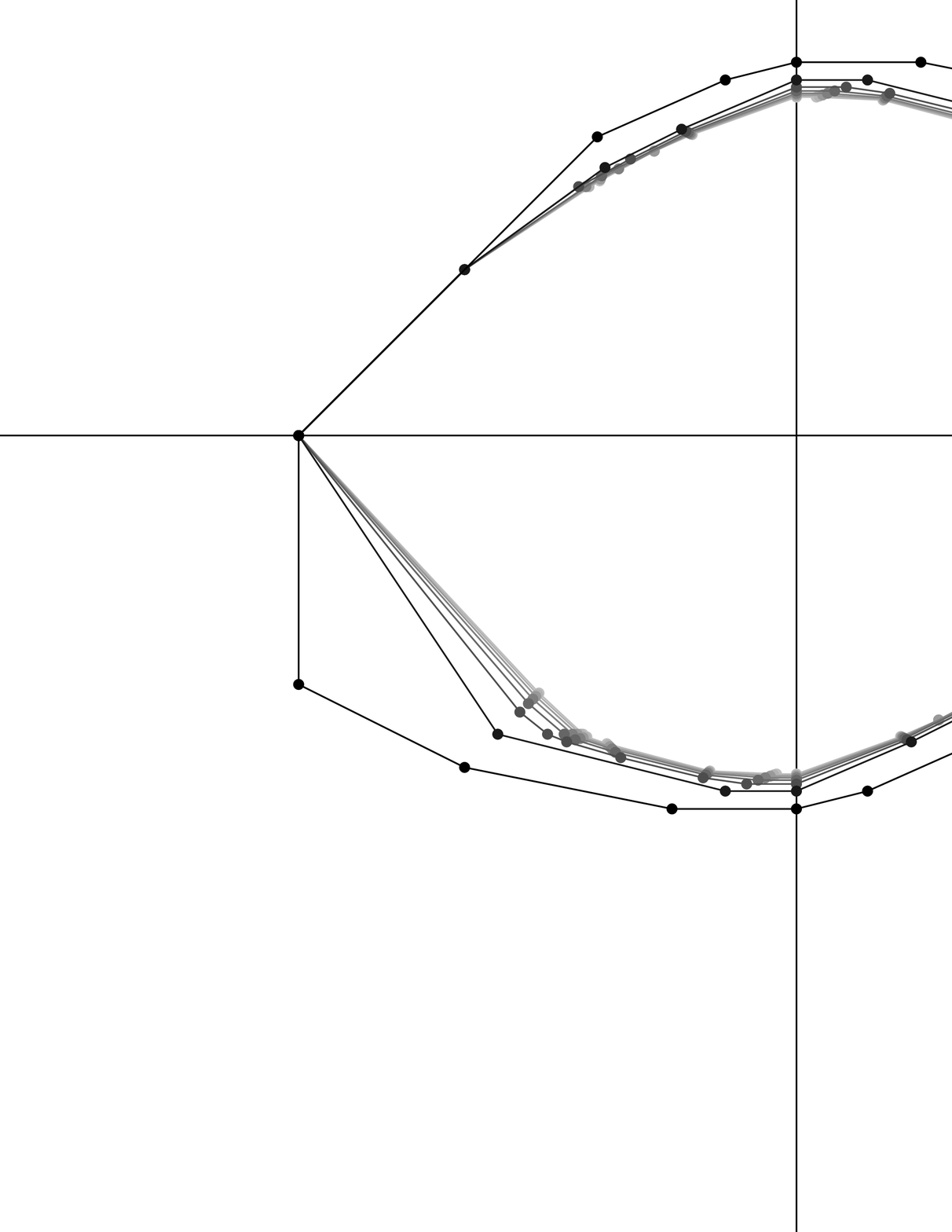}

\end{minipage}
\caption{The $\scl$ unit balls in the 2-dimensional subspaces of $B_1^H$ spanned by 
$ab^{-1}a^{-1}baba^{-1}b^{-1}$ and $aba^{-n-1}b^{-2}+a^nb$ are superimposed on one another,
for $1 \le n \le 8$. On the right is a zoomed view, showing the convergence (and the finiteness) 
of the vertices.}
\label{fig:balls}
\end{figure}

Figure \ref{fig:balls} illustrates Corollary~\ref{cor:convergent} in a particular example.  
Note the number of vertices is constant (after approximately $n=3$), and they are 
converging.  The $\scl$ unit balls were computed with the program {\tt scabble} \cite{scabble}.

\section{Acknowledgements}

Danny Calegari was supported by NSF grant DMS 1005246. We would like to thank Jesus De Loera 
and Greg Kuperberg for some useful conversations about this material.  We would also 
like to thank the anonymous referees for helpful suggestions and corrections.

\bibliography{biblio}
\bibliographystyle{plain}

\end{document}